\documentclass[a4paper, 11pt]{amsart}  

\usepackage{amssymb,amsmath,amsthm, mathtools, dsfont}
\usepackage{stmaryrd}
\usepackage{enumitem,color}
\usepackage{hyperref}
\hypersetup{hidelinks}
\usepackage[nameinlink]{cleveref}
\usepackage{nicefrac}
\usepackage{microtype}

\Crefname{theorem}{Theorem}{Theorems}
\Crefname{definition}{Definition}{Definitions}
\Crefname{lemma}{Lemma}{Lemma}
\Crefname{remark}{Remark}{Remarks}
\Crefname{conjecture}{Conjecture}{Conjectures}
\Crefname{proposition}{Proposition}{Proposition}
\Crefname{assumption}{Assump\-tion}{Assump\-tions}

\newtheorem{theorem}{Theorem}
\newtheorem{definition}[theorem]{Definition}
\newtheorem{lemma}[theorem]{Lemma}

\newtheorem{remark}[theorem]{Remark}

\newtheorem{assumption}{Assumption}

\newcommand{\e}{\mathrm{e}}
\newcommand{\C}{\mathbb{C}}

\newcommand{\R}{\mathbb{R}}

\newcommand{\iu}{\mathrm{i}}
\renewcommand{\d}{\, \mathrm{d}}
\newcommand{\Lb}{\mathcal{L}}
\renewcommand{\L}{\mathrm{L}}
\renewcommand{\H}{\mathrm{H}}
\newcommand{\W}{\mathrm{W}}

\newcommand{\scpr}[3]{\left \langle #1 , #2 \right \rangle_{#3}}
\newcommand{\dualp}[4]{\left \langle #1 , #2 \right \rangle_{#3,#4}}

\DeclareMathOperator{\real}{\mathrm{Re}}

\DeclareMathOperator{\Div}{\mathrm{div}}

\allowdisplaybreaks

\begin{document}

\title{
Input-to-state stability for bilinear feedback systems
}

\thanks{The first three authors are supported by the German Research Foundation (DFG) via the joint grant JA 735/18-1 / SCHW 2022/2-1.}
\author[R.~Hosfeld]{Ren\'e Hosfeld}
\address[RH]{University of Wuppertal, School of Mathematics and Natural Siences, IMACM, Gau\ss \-str.\ 20, 42119 Wuppertal, Germany and Department of Mathematics, University of Hamburg, Bundesstra\ss e 55, 20146 Hamburg, Germany}
\email{hosfeld@uni-wuppertal.de}
\author[B.~Jacob]{Birgit Jacob}
\address[BJ]{University of Wuppertal, School of Mathematics and Natural Sciences, IMACM, Gau\ss \-str.\ 20, 42119 Wuppertal, Germany}
\email{bjacob@uni-wuppertal.de}
\author[F.~Schwenninger]{Felix Schwenninger}
\address[FS]{Department of Applied Mathematics, University of Twente, P.O. Box 217, 7500 AE Enschede, The Netherlands and Department of Mathematics, University of Hamburg, Bundesstra\ss e 55, 20146 Hamburg, Germany}
\email{f.l.schwenninger@utwente.nl}
\author[Marius Tucsnak]{Marius Tucsnak}
\address[MT]{University of Bordeaux, Institut de Math\'ematiques de Bordeaux, 351 cours de la Lib\'eration, 33405 Talence, France}
\email{Marius.Tucsnak@math.u-bordeaux.fr}

\begin{abstract}
Input-to-state stability estimates with respect to small initial conditions and input functions for infinite-dimensional systems with bilinear feedback are shown.
We apply the obtained results to controlled versions of a viscous Burger equation with Dirichlet boundary conditions, a Schrödinger equation, a Navier--Stokes system and a semilinear wave equation. 
\end{abstract}

%
%



\maketitle
\thispagestyle{plain}
\pagestyle{plain}




\parindent0pt

\section{Introduction}

In this note we study existence of solutions, stability and robustness with respect to external inputs $u$ for abstract infinite dimensional systems of the form
\begin{equation}\label{eq:intro_BFS}
	\left\{
	\begin{aligned}
		\dot{z}(t) &= Az(t) + B_1 u(t) + B_2 N(z(t),y(t)),  \\
		z(0) &= z_0,\\
		y(t) &= Cz(t),
	\end{aligned}
	\right.
\end{equation}
where all appearing operators are possibly unbounded. The precise setting is introduced in \Cref{section:bilinear_feedback_systems}. This abstract class of evolution equations covers numerous partial differential equations such as the Navier--Stokes equation, the Burgers equation, or wave type equations, which will serve as examples for the theory developed in this note. By regarding \eqref{eq:intro_BFS} as a linear system with bilinear  feedback $N(z,y)$, we will prove that for ``small'' initial values $z_0$ and input functions $u:[0,t]\to U$, there exists a unique (mild) solution $z(\cdot)$ of \eqref{eq:intro_BFS}, which satisfies the following stability--robustness estimate
\begin{equation*}
	\lVert z(t) \rVert_X \leq k \lVert z_0 \rVert_X \e^{-\omega t} + k \e^{- \omega t} \lVert \e^{\omega \cdot} u(\cdot) \rVert_{L^2(0,t;U)},
\end{equation*}
for every $t \geq 0$ and some $\omega,k>0$, where $X$ denotes the state space and $U$ the input space. Here, smallness refers to the $X$-norm of $z_{0}$ and a weighted function norm of $u(\cdot)$.

Such an estimate is a special variant of a {\em local input-to-state stability (ISS)} estimate, where ``local'' refers to the smallness of $z_0$ and $u(\cdot)$ in a certain sense. For an overview of recent developments in ISS, the reader is referred to the survey \cite{MiroPrie20}, the book \cite{KaraKrst19book} for infinite-dimensional systems and to \cite{sontagISS} for ODE systems.

In the last decades, ISS has attracted  much attention as it has proven to be a suitable notion to study simultaneously internal stability and robustness with respect to external controls or disturbances of a system. Its various applications to nonlinear systems, such as robust stabilization, observer design and stability analysis of coupled systems and networks, have established ISS within modern control theory. The notion of ISS was originally introduced by Sontag in 1989 in \cite{Sontag89ISS} for nonlinear ODEs and input functions of class $\L^\infty$. In \cite{JNPS16}, ISS of infinite-dimensional linear systems (set $B_2=0$ in \eqref{eq:intro_BFS}) with a possibly unbounded control operator $B_1$ is studied. There, a characterization of ISS with respect to various classes of input functions, including $\L^p$-spaces for $p \in [1,\infty]$, is given in terms of exponential stability of the semigroup generated by $A$ and admissibility of $B_1$ with respect to the considered function space. Note that in infinite dimensions admissibility is a non-trivial property, which is strongly connected to the solution concept of linear systems, see for instance \cite{TucWei09} for $\L^2$-inputs. Thus, the question of ISS for the rather simple class of linear systems becomes non-trivial in infinite dimensions. Further characterizations of ISS for abstract, not necessarily linear, systems are available in terms of ISS-Lyapunov functions, provided that a global solution exists \cite{MiroPrie20}.

For nonlinear infinite dimensional systems (such as \eqref{eq:intro_BFS}), the study of ISS is much more involved and one has to face several problems which naturally come with the nonlinearity. Besides the fact that it is not possible to treat initial values and inputs separately due to the implicit form of the equation, one of the biggest issues is that global solutions do not need to exist, which is essential for ISS. The necessity to study local ISS arises by shrinking the set of initial values and control functions in order to guarantee the existence of global solutions while avoiding too strong assumptions on the nonlinearity. This difficulty, which we treat in \Cref{section:bilinear_feedback_systems} for \eqref{eq:intro_BFS}, is often ignored in the literature by assuming the existence of well behaving solutions, see e.g.~\cite{GuivLogeOpme19,MiroPrie20}.

In contrast to a variety of contributions dealing with (global) ISS (or variations of that notion), such as \cite{DaM13,DaM13b,JNPS16,MiroWirt18a} for infinite dimensional systems, \cite{JacoSchwZwar19,Krstic16MTNS,MaPr11,MiI15b,MiroKaraKrst17,Schw20,ZhZh18}  for parabolic semilinear systems, \cite{Elli09,HoJaSchw22,MiI14b} for bilinear control systems and \cite{GuivLogeOpme19,JaLoRy08} for systems with feedback laws depending only on the output, there is much less literature available on local ISS. One reason for this is that ISS is usually studied under the assumption of global (in time) solutions. In this situation, under suitable assumptions on the right-hand side, the system is ISS, see e.g.~\cite{MiroPrie20,SontagWang96}. However, note that in general \eqref{eq:intro_BFS} does  not possess  global (in time) solutions for all initial values, as can be seen by the choice $X=U_1=U_2=Y=\R$, $A=-I$, $B_1=0$, $B_2=I$, $C=I$ and $N(x,y) = xy$.

In \cite{Mir16} the author gives a characterization of local ISS in terms of so-called local ISS Lyapunov functions for a certain class of semilinear equations (without output feedback) under suitable strong assumptions, which in particular exclude unbounded control operators. A numerical algorithm for computing local ISS Lyapunov functions for a given locally ISS system is presented in \cite{LiBaGrHaWi}. In \cite{DaKaSchm20} the authors prove local ISS of a nonlinear reaction-diffusion equation. The problem of determining the local region of initial conditions and control functions for which a ODE system is local ISS is addressed in \cite{WuMeiZhang}. On the basis of these investigations, they give an algorithm for estimating local ISS properties and apply their results to a power electrics system. A comparable problem for an infinite-dimensional system appears in \cite{MirPrieWirth}, where the authors study the stabilizability of a linear unstable reaction-diffusion equation via saturated control. In this context, they give estimations for the region of attraction. In \cite{ZhengZhu_Burgerseq}, the authors proved local ISS for a Burgers' equation with in-domain and boundary disturbances under strong regularity assumptions on them.

We proceed  as follows. \Cref{section:bilinear_feedback_systems} is devoted to the main result of this paper. We give an explicit description of the bilinear feedback system \eqref{eq:intro_BFS} and we prove the existence of global solutions of \eqref{eq:intro_BFS} for small initial conditions and small input functions, as well as a local ISS result with respect to a weighted $\L^2$-norm. In \Cref{section:ISS_bilin_feedback} we present sufficient conditions for global ISS. More precisely, if either $A$ is a self-adjoint, strictly negative operator or $A=A_0+L$ with $A_0$ skew-adjoint and $L$ bounded and strictly dissipative, then global ISS estimates are shown under an additional dissipation property of the nonlinearity.  In Section \ref{section:examples} we apply the results of \Cref{section:bilinear_feedback_systems} and \Cref{section:ISS_bilin_feedback} to a controlled viscous Burgers equation with Dirichlet boundary conditions, a controlled Schrödinger equation, a controlled Navier--Stokes equation and a controlled semilinear wave equation.
\paragraph{Notation}
Let $X$ be a Banach space. We write $\lVert \cdot \rVert_X$ for the norm on $X$ and $\scpr{\cdot}{\cdot}{X}$ for the inner product, if $X$ is a Hilbert space. Let $U$ be another Banach space. If there exists a dual pairing between $X$ and $U$, we denote it by $\dualp{\cdot}{\cdot}{X}{U}$. The space of all bounded operators from $U$ to $X$ is denoted by $\Lb(U,X)$ and for $B\in \Lb(U,X)$, we write $B^*$ for both, the adjoint operator between Hilbert spaces and the dual operator between Banach spaces (with respect to the dual pairing). By $\L^p$, $\H^m$ and $\H_0^1$, we denote the standard Lebesgue and Sobolev spaces. For a function $f$ on $\R_{\geq 0} \coloneqq [0,\infty)$ we denote its restriction to $[0,t]$ for $t>0$ by $f|_{[0,t]}$.

\section{Local input-to-state stability for bilinear feedback systems}\label{section:bilinear_feedback_systems}

Let $X,U_1,U_2,Y$ be Banach spaces.
We consider the bilinear feedback system of the form
\begin{equation}\label{eq:BFS}\tag{\mbox{$\Sigma$}}
	\left\{
	\begin{aligned}
		\dot{z}(t) &= Az(t) + B_1 u_1(t) + B_2 u_2(t),  \\
		z(0) &= z_0,\\
		y(t) &= Cz(t), \\
		u_2(t)&=N(z(t),y(t)),\\
	\end{aligned}\right.
\end{equation}
where the spaces and operators satisfy the following standing assumptions
\begin{enumerate}[label=$\bullet$] 
	\item $A$ generates a $C_0$-semigroup $(T(t))_{t \geq 0}$ on $X$,
	\item $B_1 \in \Lb(U_1,X_{-1})$, $B_2 \in \Lb(U_2,X_{-1})$ are the {\em control operators},
	\item $C \in \Lb(X_1,Y)$ is the  {\em observation operator},
	\item $Y \subseteq X$ with continuous embedding and $C$ has an extension $C\in \Lb(X)$,
	\item $N:X \times Y \rightarrow U_2$ is a continuous bilinear mapping which satisfies for some $K>0$ and $p \in (0,1)$ 
	\begin{equation}\label{eq:assumptions_on_N}
		\lVert N(z,y) \rVert_{U_2} \leq K \lVert z \rVert_X \lVert y \rVert_{X}^{1-p} \lVert y \rVert_Y^p.
	\end{equation}
\end{enumerate}

The spaces $X, U_1,U_2$ and $Y$ are referred to as the {\em state space} ($X$), the {\em input spaces} ($U_1,U_2$) and the {\em output space} ($Y$).
By $X_1$ and $X_{-1}$ we denote the standard inter- and extrapolation spaces associated to the generator $A$. That is, $X_1$ is $D(A)$, the domain of $A$, equipped with the norm $\lVert x \rVert_{X_1}=\lVert (\beta -A)x \rVert_X$ and $X_{-1}$ is the completion of $X$ with respect to the norm $\lVert x \rVert_{X_{-1}}=\lVert (\beta -A)^{-1}x\rVert_X$ for some $\beta \in \rho(A)$, the resolvent set of $A$. It is well-known that these definitions are independent of the choice of $\beta \in \rho(A)$ and we have the continuous embeddings $X_1 \hookrightarrow X \hookrightarrow X_{-1}$. Furthermore, $(T(t))_{t \geq 0}$ extends uniquely to a $C_0$-semigroup $(T_{-1}(t))_{t \geq 0}$ on $X_{-1}$ whose generator, denoted by $A_{-1}$, is an extension of $A$ with domain $D(A_{-1})=X$. With the above conventions, the control operators and observation operators need not be bounded with respect to the norm of $X$, which is what is meant when we say that $B_{1}, B_{2}$ and $C$ are {\em unbounded}.
By $\omega_0((T(t))_{t \geq 0})$ we denote the growth bound of the semigroup $(T(t))_{t \geq 0}$.
The semigroup $(T(t))_{t \geq 0}$ is exponentially stable if $\omega_0((T(t))_{t \geq 0}) < 0$, i.e. there exist $M,\lambda>0$ such that
\begin{equation}\label{eq:exp_stable}
	\lVert T(t) \rVert_X \leq M \e^{- \lambda t} \quad \text{ for all } t \geq 0.
\end{equation}
Before we formalize the concept of a solution of \ref{eq:BFS} we recall some definitions and standard facts on well-posedness (in the $\L^2$-sense, see e.g. \cite{CurtainWeiss89,TucWei14,Weiss89ii,Weiss89i}) of the, to \ref{eq:BFS} associated, linear system, 
\begin{equation}\label{eq:LS}\tag{\mbox{$\Sigma_{\rm lin}$}}
	\left\{
	\begin{aligned}
		\dot{x}(t) &= Ax(t) + B_1 u_1(t) + B_2u_2(t), &&\\
		x(0) &= x_0,\\
		y(t) &= Cx(t), && \\
	\end{aligned}\right.
\end{equation}
with the same assumptions on the operators as before. The linear system \ref{eq:LS} is called well-posed, if
\begin{enumerate}[label=(\roman*)]
	\item $A$ generates a $C_0$-semigroup $(T(t))_{t\geq0}$ on $X$,
	\item for $i=1,2$, $B_i$ is an {\em $\L^2$-admissible control operator}, i.e.\ for some (and hence for all) $t>0$, it holds that the {\em input map} $\Phi_t^i$ is an operator in $\Lb(\L^2(0,\infty;U_i),X)$, where
	\begin{equation*}
		\Phi_t^i u_i \coloneqq \int_0^t T_{-1}(t-s) B_i u_i(s) \d s, \quad u_i \in \L^2(0,\infty;U_i).
	\end{equation*}
	\item $C$ is an {\em $\L^2$-admissible observation operator}, i.e.\ for some (and hence for all) $t>0$, it holds that the {\em output map} $\Psi_t$ extends to an operator in $\Lb(X,\L^2(0,\infty,Y))$, where
	\begin{equation*}
		\Psi_t x = (CT(\cdot)x)|_{[0,t)} \quad x \in D(A).
	\end{equation*}
	\item for $i=1,2$, there exists a family $(\mathbb{F}_t^i)_{t \geq 0}$ of bounded linear operators $\mathbb{F}_t^i$ from $\L^2(0,\infty;U_i)$ to $\L^2(0,\infty;Y))$, called the {\em input-output maps}, such that $\mathbb{F}_0^i=0$ and
	\begin{align*}
		\mathbb{F}_{\tau+t}^i (u \underset{\tau}{\diamond} v)= \mathbb{F}_\tau^i u \underset{\tau}{\diamond} (\Psi_t \Phi_\tau^i u + \mathbb{F}_t^i v),
	\end{align*}
	for all $\tau,t>0$ and $u,v \in \L^2(0,\infty;U_i)$, where
	\begin{equation*}
		(u \underset{\tau}{\diamond} v)(t) \coloneqq \left\{
		\begin{aligned}
			& u (t), && \text{if } t \in [0,\tau),\\
			& v(t- \tau), && t \in [\tau,\infty).
		\end{aligned}
		\right.
	\end{equation*}
\end{enumerate}

If $U_i$, $X$ and $Y$ are Hilbert spaces, \cite[Theorem~5.1]{CurtainWeiss89} states that (iv) can be equivalently replaced by
\begin{enumerate}
	\item[(iv')] Some (and hence any) function $G_i:\C_{\omega_0( (T(t))_{t \geq 0})} \rightarrow \Lb(U_i,Y)$, $i=1,2$, satisfying
	\begin{equation*}
		G_i(\alpha) - G_i(\beta) = C[(\alpha-A_{-1})^{-1} - (\beta - A_{-1})^{-1}]B_i
	\end{equation*}
	is bounded on some right half-plane $\C_\gamma = \{ \lambda \in \C \mid \real \lambda > \gamma\}$.
\end{enumerate}
Note that the right-hand side of the expression above is well-defined by the resolvent identity.
The functions $G_i$, $i=1,2$, are called the {\em transfer functions} associated to the triples $(A,B_i,C)$, $i=1,2$, and they are unique up to a constant operator. Given such a transfer function, the operators $\mathbb{F}_t^i:\L^2(0,t;U) \rightarrow \L^2(0,t;Y)$ given by
\begin{equation*}
	(\mathbb{F}_t^i u_i)(\cdot) =  C\left[ \int_0^\cdot T_{-1}(\cdot -s) B_iu_i(s) \d s - (\alpha -A_{-1})^{-1}B_iu_i(\cdot) \right]+ G_i(\alpha) u_i(\cdot), 
\end{equation*}
for $u_i \in W_{\rm loc}^{1,2}(0,\infty;U_i)$ with $u_i(0)=0$ and $\real \alpha$ large enough extend to a family $(\mathbb{F}_t^i)_{t \geq 0}$ satisfying (iv), see~\cite{Weiss88,CurtainWeiss89}. Here, $\W_{\rm loc}^{1,2}$ denotes the standard Sobolev space consisting of $\L^2_{\rm loc}$-functions whose weak derivatives exist in a distributional sense and are again $\L^2_{\rm loc}$-functions. Since $\W_{\rm loc}^{1,2}$ is dense in $\L^2$ on bounded intervals, the operators $\mathbb{F}_t^i$ are uniquely determined by that formula. Moreover, the family $(\mathbb{F}_t^i)_{t \geq 0}$ is uniquely determined by $A$, $B$ and $C$ up to an additive constant operator related to the transfer function.

Functions $x \in C([0,\infty);X)$ and $y \in \L^2_{\rm loc}(0,\infty;Y)$ are called a {\em solution (or state trajectory)} and the {\em output function (or output)} of \ref{eq:LS}, respectively, with input data $x_0 \in X$ and $u_i \in \L^2(0,\infty;U_i)$, $i=1,2$, if they satisfy the equations
\begin{align*}
	x(t) &= T(t)x_0 + \Phi_t^1u_1 + \Phi_t^2u_2,\\
	y|_{[0,t]}&= \Psi_t x_0+ \mathbb{F}_t^1 u_1 + \mathbb{F}_t^2u_2
\end{align*}
for all $t\ge0$.
From this formula it follows immediately that a well-posed linear system admits a unique solution $x$ and output $y$ for all $x_0 \in X$ and $u_i \in \L^2(0,\infty;U_i)$ and for $t\geq 0$ there exist positive constants $k_{1,t}$ and $k_{2,t}$ such that for all $x_0 \in X$ and $u_i \in \L^2(0,\infty;U_i)$ it yields
\begin{align}\label{eq:LS_well_posed}
	\begin{split}
		\lVert x(t) \rVert_X &\leq k_{1,t} (\lVert x_0 \rVert_X + \lVert u_1 \rVert_{\L^2(0,t;U_1)} + \lVert u_2 \rVert_{\L^2(0,t;U_2)} ),\\
		\lVert y \lVert_{\L^2(0,t;Y)} &\leq k_{2,t} (\lVert x_0 \rVert_X + \lVert u_1 \rVert_{\L^2(0,t;U_1)} + \lVert u_2 \rVert_{\L^2(0,t;U_2)} ).
	\end{split}
\end{align}
If $A$ generates an exponentially stable semigroup, then $k_{1,t}$ and $k_{2,t}$ can be chosen independent of $t$, see e.g.~\cite{Weiss89i, Weiss89ii, CurtainWeiss89}.

\begin{definition}
	Let \ref{eq:LS} be well-posed. 
	The functions $z \in C([0,\infty);X)$ and $y \in \L^2_{\rm loc}(0,\infty;Y)$ are the {\em solution} and {\em output} of \ref{eq:BFS} with input data $z_0 \in X$ and $u_1 \in \L^2(0,\infty;U_1)$, if they satisfy
	\begin{align}\label{eq:bilin_feedback_solution_and_output}
		\begin{split}
			z(t) &= T(t)z_0 + \Phi_t^1u_1 + \Phi_t^2N(z,y),\\
			y|_{[0,t]}&= \Psi_t z_0+ \mathbb{F}_t^1 u_1 + \mathbb{F}_t^2 N(z,y).
		\end{split}
	\end{align}
\end{definition}

For $t_0 \in (0,\infty]$ and $\omega>0$, let  $\L^2_\omega(0,t_0;U)$ be the weighted $\L^2$-space of $U$-valued functions given by
\begin{equation*}
	\L^2_\omega(0,t_0;U) \coloneqq \{ f \in \L^2(0,t_0;U) \mid \e^{\omega \cdot} f \in \L^2(0,t_0;U) \}
\end{equation*}
with norm 
\begin{equation*}
	\lVert f \rVert_{\L^2_\omega(0,t_0;U)} = \lVert \e^{\omega \cdot} f \rVert_{\L^2(0,t_0;U)}.
\end{equation*}

We introduce the notion of (local) input-to-state stability for \ref{eq:BFS} with respect to $\mathcal{U}=\L^p$ for $1 \leq p \leq \infty$ and $\mathcal{U}=\L^2_\omega$ for $\omega >0$. It is obvious that the definition can be generalized to spaces of input functions $\mathcal{U}$ different from $\L^2$ and $\L^2_{\omega}$ and to other (abstract) systems than \ref{eq:BFS} without any problems, see e.g.~\cite{JNPS16,MiroPrie20}. By $\mathcal{KL}$ and $\mathcal{K}_\infty$ we denote the classes of comparison functions defined by
\begin{align*}
	\mathcal{KL} &\coloneqq  \left\{ \beta \in C(\R_{\geq 0}^2;\R_{\geq 0}) \middle \vert \,
	\begin{aligned}
		&\forall t \geq 0: \beta(\cdot,t) \text{ is  strictly increasing} \\
		&\text{with } \beta(0,t)=0 \text{ and }\\
		&\forall s>0: \beta(s,\cdot) \text{ is strictly decreasing to } 0
	\end{aligned}
	\right\},\\
	\mathcal{K}_\infty & \coloneqq \left\{ \gamma \in C(\R_{\geq 0};\R_{\geq 0}) \middle \vert \, \gamma \text{ is strictly increasing to } \infty \text{ with } \gamma(0)=0 \right\}.
\end{align*}

\begin{definition}
	Let $\mathcal{U}$ be either $\L^p(0,\infty;U)$ for $1 \leq p \leq \infty$ or $\L^2_\omega(0,\infty;U)$ for some $\omega >0$.
	The bilinear feedback system \ref{eq:BFS} is called {\em locally input-to-state stable with respect to $\mathcal{U}$ ($\mathcal{U}$-ISS)} if there exist $\beta \in \mathcal{KL}$, $\gamma \in \mathcal{K}_\infty$ and $\varepsilon>0$ such that \ref{eq:BFS} admits for all $z_0 \in X$ and $u_1 \in \mathcal{U}$ with $\lVert z_0 \rVert_X + \lVert u \rVert_{\mathcal{U}} \leq \varepsilon$ a unique solution $z \in C([0,\infty);X)$ which satisfies for every $t \geq 0$
	\begin{equation}\label{eq:mathcalU_ISS_definition}
		\lVert z(t) \rVert_X \leq \beta(\lVert z_0 \rVert_X, t) + \gamma(\lVert u_1|_{[0,t]} \rVert_{\mathcal{U}}).
	\end{equation}
	System \ref{eq:BFS} is called {\em $\mathcal{U}$-ISS} if the above holds with $\varepsilon= \infty$.
\end{definition}

In \cite[Section~7]{TucWei14} the authors proved existence of local in time solutions under our general assumptions on the bilinear feedback system. We modified their techniques to obtain global in time solutions as well as an $\L^2_\omega$-ISS estimate.

\begin{theorem}\label{thm:global_existence_for_small_data} 
	Let $A$ be the generator of an exponentially stable semigroup $(T(t))_{t \geq 0}$ with $M,\lambda>0$ such that \eqref{eq:exp_stable} holds, and let \ref{eq:LS} be well-posed. Then for every $\omega \in (0, \lambda)$ there exists a constant $\varepsilon>0$ such that for all input data $z_0 \in X$ and $u_1 \in \L^2_\omega(0,\infty;U_1)$ with
	\begin{equation}\label{eq:small_initial_data}
		\lVert z_0 \rVert_X + \lVert u_1 \rVert_{\L^2_\omega(0,\infty;U_1)} \leq \varepsilon,
	\end{equation}
	it holds that \ref{eq:BFS} admits a unique solution $z \in C([0,\infty);X)$ and an output $y \in \L^2(0,\infty;Y)$. Furthermore, there exists a constant $k>0$ such that for every $t \geq 0$
	\begin{equation}\label{eq:z_exp_bounded}
		\lVert z(t) \rVert_X \leq  k \e^{-\omega t} (\lVert z_0 \rVert_X + \lVert u_1 \rVert_{\L^2_\omega(0,t;U_1)}).
	\end{equation}
	In particular, \ref{eq:BFS} is locally $\L^2_\omega$-ISS.
\end{theorem}

\begin{proof}
	Let $\omega \in (0,\lambda)$, $z_0 \in X$ and $u_i \in \L^2_\omega(0,\infty;U_i)$ for $i=1,2$. By $z$ and $y$ we denote the state trajectory and the output of \ref{eq:LS}. Since $\e^{\omega \cdot} u_i \in \L^2(0,\infty;U_i)$ for $i=1,2$, the functions $x=\e^{\omega \cdot} z$ and $\e^{\omega \cdot}y$ are the state trajectory and the output of the shifted well-posed linear system
	\begin{equation}\label{eq:shiftedLS}\tag{\mbox{$\widetilde{\Sigma}_{\rm lin}$}}
		\left\{
		\begin{aligned}
			\dot{x}(t) &= (A+\omega I)x(t) + B_1 \e^{\omega t}u_1(t) + B_2 \e^{\omega t} u_2(t), && t \geq 0, \\
			x(0) &= z_0,\\
			\e^{\omega t} y
			(t) &= C x(t), && t \geq 0.\\
		\end{aligned}\right.
	\end{equation}
	By our choice of $\omega$, $A+ \omega I$ generates an exponentially stable semigroup. Thus, by \eqref{eq:LS_well_posed}, once applied to the shifted linear system (for the state) and once to the original linear system (for the output) there exist constants $k_1,k_2>0$ such that
	\begin{align}\label{eq:z_exp_bounded_proof}
		\begin{split}
			\lVert \e^{\omega t}z(t) \rVert_X &\leq  k_1  (\lVert z_0 \rVert_X + \lVert \e^{\omega \cdot} u_1 \rVert_{\L^2(0,\infty;U_1)} + \lVert \e^{\omega \cdot} u_2 \lVert_{\L^2(0,\infty;U_2)} ),\\
			\lVert y \rVert_{\L^2(0,t;Y)} & \leq k_2  (\lVert z_0 \rVert_X + \lVert u_1 \rVert_{\L^2(0,\infty;U_1)} + \lVert u_2 \lVert_{\L^2(0,\infty;U_2)} )
		\end{split}
	\end{align}
	holds.
	
	Let $K \geq 0$ and $p \in (0,1)$ such that \eqref{eq:assumptions_on_N} is satisfied and choose
	$\varepsilon>0$ such that
	\begin{align}\label{eq:choice_varepsilon}
		\begin{split}
			\varepsilon &\leq  \frac{(2 \omega)^{\nicefrac{(1-p)}{2}}}{4 K \lVert C \rVert_{\Lb(X)}^{1-p} k_1^{2-p} k_2^p}, \\
			\varepsilon &<\frac{(2 \omega)^{\nicefrac{(1-p)}{2}}}{2 K \lVert C \rVert_{\Lb(X)}^{1-p} k_1^{1-p} k_2^p}.
		\end{split}
	\end{align}
	
	Now, take $z_0 \in X$ and $u_1 \in \L^2_\omega(0,\infty;U_1)$ such that \eqref{eq:small_initial_data} holds with $\varepsilon$ as above. For any $u_2$ in the set
	\begin{equation*}
		S_\varepsilon \coloneqq \{ u \in \L^2_\omega(0,\infty;U_2) \mid \lVert u \rVert_{\L^2_\omega(0,\infty;U_2)} \leq \varepsilon \}
	\end{equation*}
	we denote by $z$ and $y$ the state trajectory and the output of the linear system \ref{eq:LS} with input data $z_0$ and $u_i$, $i=1,2$. We will prove that 
	\begin{equation*}
		\mathcal{G}: S_\varepsilon \rightarrow S_\varepsilon, \quad \mathcal{G}(u_2) \coloneqq N(z,y)
	\end{equation*}
	is a contraction.
	This then implies that $\mathcal{G}$ admits a unique fixed point in $S_\varepsilon$ and thus, existence of a solution follows. Uniqueness of the solution follows from \cite[Theorem~7.6]{TucWei14}.\\
	In order to verify our claim on $\mathcal{G}$, we first check that $\mathcal{G}$ is well-defined. From our assumptions on $N$ and the boundedness of $C$, we deduce for almost every $t>0$ that
	\begin{align*}
		\lVert \e^{\omega t} N(z(t),y(t)) \rVert_{U_2}
		&\leq K \e^{\omega t} \lVert z (t) \rVert_X \lVert y(t) \rVert_X^{1-p} \lVert y(t) \rVert_{Y}^p\\
		&\leq K\|C\|_{\mathcal{L}(X)}^{1-p}   \lVert \e^{\omega t}z(t) \rVert_X^{2-p}  \e^{(p-1)\omega t}\lVert y(t) \rVert_{Y}^p\\
		&\leq  K\|C\|_{\mathcal{L}(X)}^{1-p}  (2k_{1}\varepsilon)^{2-p} \e^{(p-1) \omega t} \lVert y(t) \rVert_{Y}^p,
	\end{align*}
	where the last inequality holds by the first inequality in \eqref{eq:z_exp_bounded_proof}, \eqref{eq:small_initial_data} and since $u_2 \in S_\varepsilon$. 
	We infer by H\"older's inequality, \eqref{eq:z_exp_bounded_proof}, \eqref{eq:small_initial_data}, $u_2 \in S_\varepsilon$ and \eqref{eq:choice_varepsilon} that
	\begin{align*}
		\lVert \e^{\omega \cdot} N(z,y) \rVert_{\L^2(0,\infty;U_2)}
		&\leq K \lVert C \rVert_{\Lb(X)}^{1-p} (2k_1 \varepsilon)^{2-p} \left(\frac{1}{2 \omega}\right)^{\nicefrac{(1-p)}{2}} \lVert y \rVert_{\L^2(0,\infty;Y)}^p \\
		&\leq K\|C\|_{\Lb(X)}^{1-p} k_{1}^{2-p} k_2^p \left(\frac{1}{2 \omega}\right)^{\nicefrac{(1-p)}{2}} (2\varepsilon)^2 \leq \varepsilon
	\end{align*} 
	and thus, $\mathcal{G}(u_{2})\in S_{\varepsilon}$.\\
	We conclude contractivity using similar methods. Let $v_i\in S_\varepsilon$, $i=1,2$, be arbitrary. By $z_i$ and $y_i$, $i=1,2$, we denote the state trajectory and the output of \ref{eq:LS} with input data $z_0$ and $u_1$ satisfying \eqref{eq:small_initial_data} and $u_2=v_i$, $i=1,2$. Since $N$ is bilinear, it holds that
	\begin{align*}
		\mathcal{G}(v_1) - \mathcal{G}(v_2)
		&= N(z_1,y_1) - N(z_2,y_2)\\
		&= N(z_1-z_2,y_1) + N(z_2,y_1-y_2).
	\end{align*}
	We estimate each term separately. Note that $\e^{\omega \cdot}(z_1-z_2)$ is the state trajectory of the shifted linear system \ref{eq:shiftedLS} with $z_0=0$, $u_1=0$ and $u_2= v_1-v_2$. From \eqref{eq:z_exp_bounded_proof}, and as before, the boundedness of $C$, \eqref{eq:small_initial_data} and the fact that $v_1 \in S_\varepsilon$, we deduce that
	\begin{align*}
		\MoveEqLeft \lVert \e^{\omega t} N(z_1(t)-z_2(t),y_1(t)) \rVert_{U_2}\\
		&\leq K \lVert \e^{\omega t} (z_1(t) - z_2(t)) \rVert_X \lVert \e^{\omega t}y_1(t) \rVert_X^{1-p} \e^{(p-1)\omega t}\lVert y_1(t) \rVert_Y^p\\
		& \leq K \lVert C \rVert_{\Lb(X)}^{1-p} k_1^{2-p} (2 \varepsilon)^{1-p} \e^{(p-1)\omega t} \lVert y (t) \rVert_Y^p \lVert \e^{\omega \cdot} (v_1 -v_2) \rVert_{\L^2(0,\infty;U_2)}.
	\end{align*}
	Applying H\"older's inequality and \eqref{eq:z_exp_bounded_proof}, as before, yields
	\begin{align*}
		\MoveEqLeft \lVert \e^{\omega \cdot} N(z_1-z_2,y_1) \rVert_{\L^2(0,\infty;U_2)}\\
		&\leq 2 K \lVert C \rVert_{\Lb(X)}^{1-p} k_1^{2-p} k_2^p \varepsilon \lVert\e^{\omega \cdot} (v_1-v_2) \rVert_{\L^2(0,\infty;U_2)}.
	\end{align*}
	For the second term we obtain similarly
	\begin{align*}
		\MoveEqLeft \lVert \e^{\omega t} N(z_2(t),y_1(t)-y_2(t)) \rVert_{U_2}\\
		&\leq 2K \lVert C \rVert_{\Lb(X)}^{1-p} k_1^{1-p}\varepsilon \lVert \e^{\omega \cdot}(v_1 -v_2) \rVert_{\L^2(0,\infty;U_2)}^{1-p} \e^{(p-1) \omega t} \lVert y_1(t)- y_2(t) \rVert_{Y}^p.
	\end{align*}
	Again, H\"older's inequality and \eqref{eq:z_exp_bounded_proof} yield that
	\begin{align*}
		\MoveEqLeft \lVert \e^{\omega \cdot} N(z_2,y_1-y_2) \rVert_{\L^2(0,t;U_2)}\\
		&\leq K \lVert C \rVert_{\Lb(X)}^{1-p} k_1^{2-p} k_2^p \left( \frac{1}{2 \omega} \right)^{\nicefrac{(1-p)}{2}}  \varepsilon \lVert \e^{\omega \cdot} (v_1-v_2)\rVert_{\L^2(0,\infty;U_2)}.
	\end{align*}
	Hence, it follows that
	\begin{align*}
		\MoveEqLeft \lVert \e^{\omega \cdot} (\mathcal{G}(v_1) - \mathcal{G}(v_2)) \rVert_{\L^2(0,\infty;U_2)} \\
		&\leq \lVert \e^{\omega \cdot} N(z_1 - z_2 , y_1) \rVert_{\L^2(0,\infty;U_2)} + \lVert \e^{\omega \cdot} N(z_2, y_1-y_2) \rVert_{\L^2(0,\infty;U_2)} \\
		& \leq 4 K \lVert C \rVert^{1-p} k_1^{2-p} k_2^p \left(\frac{1}{2 \omega} \right)^{\nicefrac{1-p}{2}} \varepsilon \lVert \e^{\omega \cdot} (v_1 - v_2) \rVert_{\L^2(0,\infty;U_2)}.
	\end{align*}
	By \eqref{eq:choice_varepsilon}, $\mathcal{G}$ is a contraction on $S_\varepsilon$ and therefore, there exists a unique $u_2 \in S_\varepsilon$ such that $u_2=N(z,y)$, where $z$ and $y$ are the solution and the output of \ref{eq:LS} with input data $x_0$ and $u_1$ satisfying \eqref{eq:small_initial_data}. Hence, $z$ and $y$ are the solution and the output of \ref{eq:BFS} and from \eqref{eq:z_exp_bounded_proof} we deduce that 
	\begin{equation}\label{eq:proof_hidden_ISS_estimate}
		\lVert z(t) \rVert_X \leq 2 k_1 \varepsilon \e^{-\omega t}.
	\end{equation}
	To prove the ISS estimate, let $\varepsilon$ be given as above and let $z_0 \in X$ and $u_1 \in \L^2_\omega(0,\infty;U_2)$ such that \eqref{eq:small_initial_data} holds and denote the corresponding solution and output of \ref{eq:BFS} by $z$ and $y$, respectively. Further, let $t>0$ be arbitrary and define
	\begin{equation}
		\tilde{\varepsilon} \coloneqq \lVert z_0 \rVert_X + \lVert u_1 \rVert_{\L^2_\omega(0,t;U_1)} \leq \varepsilon.
	\end{equation}
	It is clear that $\tilde{\varepsilon}$ satisfies \eqref{eq:choice_varepsilon} and that \ref{eq:BFS} admits for $z_0$ and $\tilde{u}_1=\mathds{1}|_{[0,t]} u_1$ a unique solution $\tilde{z}$ satisfying \eqref{eq:proof_hidden_ISS_estimate} with $\tilde{\varepsilon}$, i.e. by the definition of $\tilde{\varepsilon}$
	\begin{equation*}
		\lVert \tilde{z}(t) \rVert_X \leq 2k_1 \e^{-\omega t} (\lVert z_0 \rVert_X + \lVert u_1 \rVert_{\L^2_\omega(0,t;U_1)}).
	\end{equation*}
	
	As a consequence of the causality of the linear system \ref{eq:LS}, we obtain that
	\begin{equation*}
		z|_{[0,t]} = \tilde{z}|_{[0,t]},
	\end{equation*}
	which completes the proof.
\end{proof}

\begin{remark}    \Cref{thm:global_existence_for_small_data} also holds for $\L^q$-well-posed linear systems with $1 \leq q <\infty$ \ref{eq:LS} (see~\cite[Definition~2.2.1]{Staffans} for a definition) with exponentially stable $C_0$-semigroup. In this case, we consider input functions $u_1 \in \L^q_\omega$, the weighted $\L^q$-space, defined analogously to $\L^2_\omega$, and obtain an $\L^q_\omega$-ISS estimate under analog smallness condition as before. The proof stays the same up to adaption of the used H\"older inequalities and the resulting constants.
\end{remark}

\begin{remark}\label{rem:standard_norms_but_not_standrd_ISS}
	In the situation of \Cref{thm:global_existence_for_small_data} we obtain as a direct consequence of \eqref{eq:z_exp_bounded} and H\"older's inequality that
	\begin{equation*}
		\lVert z(t) \rVert_X \leq k \lVert z_0 \rVert_X \e^{- \omega t } + k \lVert u_1 \rVert_{\L^q(0,t;U_1)}
	\end{equation*}
	holds for every $t \geq 0$, $q \in [2,\infty]$ and $u_1 \in \L^q(0,\infty;U) \cap \L^2_\omega(0,\infty;U)$. This is an $\L^p$-ISS-estimate, however, this does not mean that \ref{eq:BFS} is locally $\L^q$-ISS, since the equation only holds for small input functions in the intersection of $\L^q$ with the weighted space $\L^2_\omega$.
\end{remark}

\begin{remark}
	Condition \eqref{eq:choice_varepsilon} shows how $\varepsilon$ in \Cref{thm:global_existence_for_small_data} can be chosen, depending on the decay rate $\omega$, the constants $\lVert C \rVert_{\Lb(X)}$, $k_1$ and $k_2$ corresponding to the linearized system via \eqref{eq:z_exp_bounded_proof} and the constants $K$ and $p$ corresponding to nonlinearity $N$ via \eqref{eq:assumptions_on_N}. Condition \eqref{eq:choice_varepsilon} is not optimal for specific systems (see e.g.~\Cref{thm:Burgers_global}). However, it can be slightly improved by using sharper estimates than \eqref{eq:z_exp_bounded_proof} for the linear system involving the operator norms of the semigroup, the input maps, the output maps and the input-output maps.
\end{remark}

\begin{remark}
	Although the assumption that $A$ generates an exponentially stable $C_0$-semigroup in \Cref{thm:global_existence_for_small_data} may sound restrictive, it naturally appears when studying ISS of such bilinear systems and cannot be weakened in the view of our general abstract setting. Indeed, the trivial choices $N=0$ and $C=0$ lead to the linear system \ref{eq:LS} with $B_2=0$ from which we know that ISS (or, equivalently, local ISS) requires exponential stability of the semigroup. However, for particular nonlinearities, it might be possible  to  weaken the assumption on the semigroup. 
\end{remark}

\section{Global input-to-state stability for bilinear feedback systems} \label{section:ISS_bilin_feedback}

Throughout this section we assume, in addition to our standing assumptions from \Cref{section:bilinear_feedback_systems}, that $U_1,U_2,X,Y$ are Hilbert spaces.

We consider the following two assumptions:

\begin{assumption}\label{assump:A_self-adjoint}
	The operator $A$ is self-adjoint and strictly negative, i.e. there exists a constant $w_A<0$ such that
	\begin{equation*}
		\scpr{Az}{z}{X} \leq w_A \lVert z \rVert_{X}^2 \quad \text{ for all } z \in D(A),
	\end{equation*}
	$B_i \in \Lb(U_i;X_{-\frac{1}{2}})$ and $C \in \Lb(X_{\frac{1}{2}},Y)$, where the spaces $X_{\frac{1}{2}}$ and $X_{-\frac{1}{2}}$ are introduced below.
\end{assumption}

\begin{assumption}\label{assump:A_essentially_skew-adjoint}
	The operator $A$ is of the form $A=A_0 + L$, where $A_0$ is skew-adjoint and $L \in \Lb(X)$ is strictly dissipative, i.e. there exists a constant $w_A<0$ such that
	\begin{equation*}
		\real \scpr{Lz}{z}{X} \leq w_A \lVert z \rVert_{X}^2 \quad \text{ for all } z \in X,
	\end{equation*}
	$B_i \in \Lb(U_i,X)$, $i=1,2$, and $C \in \Lb(X,Y)$.
\end{assumption}

\begin{remark}
	Under \Cref{assump:A_essentially_skew-adjoint} we have for $z \in D(A)$ that 
	\begin{equation*}
		\real \scpr{Az}{z}{X} = \real \scpr{Lz}{z}{X} \leq w_A \lVert z \rVert_X^2.
	\end{equation*}
	Thus, it makes sense to denote the constant by $w_A$. In particular, $A$ is strictly dissipative and thus the generator of an exponentially stable $C_0$-semigroup.
\end{remark}

If $A$ is a self-adjoint and strictly negative operator on $X$, then it follows from \cite[Chapter~II,~Corollary~4.7 \& Chapter~IV,~Corollary~3.12]{engelnagel} and the fact that $\sigma(A) \subset (-\infty,w_A]$ holds that  $A$ is the generator of an exponentially stable, analytic semigroup on $X$.
One can define (c.f.~\cite[Section~3.4]{TucWei09}) the spaces $X_{\frac{1}{2}}$ as the completion of $D(A)$ with respect to the norm
\begin{equation}\label{eq:1/2_norm}
	\lVert z \rVert_{X_\frac{1}{2}} =  \sqrt{\scpr{(-A)^{\nicefrac{1}{2}}z}{(-A)^{\nicefrac{1}{2}}z}{X}}=\sqrt{\scpr{-Az}{z}{X}}, \quad z \in D(A)
\end{equation}
and $X_{-\frac{1}{2}}$ as its dual space with respect to the pivot space $X$.
We have the continuous and dense embeddings
\begin{equation*}
	X_1 \hookrightarrow X_{\frac{1}{2}} \hookrightarrow X \hookrightarrow X_{-\frac{1}{2}} \hookrightarrow X_{-1}.
\end{equation*}
From the concept of dual spaces with respect to a pivot space and the fact that $A$ is self-adjoint we obtain a natural continuous dual pairing $\dualp{\cdot}{\cdot}{X_{-\frac{1}{2}}}{X_{\frac{1}{2}}}: X_{-\frac{1}{2}} \times X_{\frac{1}{2}} \rightarrow \C$ which simplifies for $z \in X$ and $y \in X_{\frac{1}{2}}$ to
\begin{equation*}
	\dualp{z}{y}{X_{-\frac{1}{2}}}{X_{\frac{1}{2}}} = \scpr{z}{y}{X}.
\end{equation*}
Further, it is easy to see that $A_{-1}:X_{\frac{1}{2}} \rightarrow X_{-\frac{1}{2}}$ is isometric. This allows to extend \eqref{eq:1/2_norm} and the strict negativity inequality to
\begin{equation}\label{eq:A_strictly_negative_extended}
	\dualp{A_{-1}z}{z}{X_{-\frac{1}{2}}}{X_{\frac{1}{2}}} = -\lVert z \rVert_{X_{\frac{1}{2}}}^2 \leq w_A \lVert z \rVert_X^2,
\end{equation}
for every $z \in X_{\frac{1}{2}}$.

\begin{remark}
	We emphasize that \Cref{assump:A_self-adjoint} and \Cref{assump:A_essentially_skew-adjoint} guarantee that $A$ generates an exponentially stable $C_0$-semigroup and that \ref{eq:LS} is well-posed, see for instance \cite[Section~6]{TucWei14}. In particular, \Cref{thm:global_existence_for_small_data} is applicable under these assumptions.
\end{remark}

We will give an additional dissipation condition of the non-linearity to obtain a global ISS result for \ref{eq:BFS}. Before that, we briefly recall the result \cite[Theorem~7.6]{TucWei14} on the existence of a local in time solution for \ref{eq:BFS} and prove some natural regularity properties of this solution and its squared norm under \Cref{assump:A_self-adjoint} and \Cref{assump:A_essentially_skew-adjoint}.

\begin{lemma}\label{lem:local_in_time_solution}
	Suppose that $U_1,U_2,X,Y$ are Banach spaces and that \ref{eq:LS} is well-posed and that our standing assumptions on \ref{eq:BFS} hold.  Then, for every $z_0 \in X$ and $u_1 \in \L^2_{\rm loc}(0,\infty;U_1)$ there exists $t_1>0$ such that \ref{eq:BFS} admits on $[0,t_1]$ a unique solution $z \in C([0,t_1];X)$ and output $y \in \L^2(0,t_1,Y)$, that is, $z$ and $y$ satisfy \eqref{eq:bilin_feedback_solution_and_output} for every $t \in [0,t_1]$. For given $z_0$ and $u_1$, denote by $t_{\rm max}$ the supremum over all $t_1$ such that \ref{eq:BFS} admits a solution on $[0,t_1]$. If $t_{\rm max}$ is finite, then
	\begin{equation*}
		\lim_{t \nearrow t_{\rm max}} \lVert z(t) \rVert_X = \infty.
	\end{equation*}
\end{lemma}

\begin{proof}
	The reader is referred to \cite[Remark~7.5~\&~Theorem~7.6]{TucWei14} for a proof.
\end{proof}

\begin{lemma}\label{lem:representation_norm(z)^2_analytic}
	Let $U_1,U_2,X,Y$ be Hilbert spaces. In addition to our standing assumptions, suppose that \Cref{assump:A_self-adjoint} holds. For $z_0 \in X$ and $u_1 \in \L^2_{\rm loc}(0,\infty;U_1)$ let $[0,t_{\rm max})$ be the maximal existence interval of the solution $z$  of \ref{eq:BFS}. Then, it holds that $z\in H^1(0,t_{\rm max};X_{-\frac{1}{2}}) \cap C([0,t_{\rm max});X) \cap \L^2(0,t_{\rm max};X_{\frac{1}{2}})$ and
	\begin{align}\label{eq:representation_norm(z)^2_analytic}
		\begin{split}
			\MoveEqLeft \lVert z(t) \rVert_X^2 - \lVert z_0 \rVert_X^2 \\
			&= 2 \real \int_0^t   \dualp{Az(s)}{z(s)}{X_{-\frac{1}{2}}}{X_\frac{1}{2}}  + \scpr{u_1(s)}{B_1^*z(s)}{U_1} \\
			& \qquad \qquad \qquad \qquad \qquad + \scpr{N(z(s),Cz(s))}{B_2^*z(s)}{U_2} \d s.
		\end{split}
	\end{align}
	for all $t \in [0,t_{\rm max})$. 
\end{lemma}

\begin{proof}
	Since we are dealing with an exponentially stable analytic semigroup on a Hilbert space, its maximal regularity property (see e.g.~\cite[Part~II,~Chapter~1,~Example~4.3 \& Chapter~3,~Theorem~2.1]{BenPratDefMit}) yields for the linear system with input data $z_0 \in X$ and $u_i \in \L^2(0,\infty;U_i)$, $i=1,2$, a solution $z \in \H^1(0,\infty;X_{-\frac{1}{2}}) \cap C([0,\infty);X) \cap \L^2(0,\infty;X_\frac{1}{2})$. Moreover, $z$ depends continuously in $C([0,t];X)$ and $\L^2(0,t;X_{\frac{1}{2}})$ on the input data $z_0 \in X$ and $u_i \in \L^2(0,t;U_i)$, $i=1,2$, for every $t>0$. This implies that $A_{-1}z$ depends continuously in $\L^2(0,t;X_{-\frac{1}{2}})$ on the input data. \\
	Now, let $z_0 \in X_1$ and $u_i \in \H^1(0,\infty;U_i)$, $i=1,2$, with $u_i(0)=0$. Thus, $z$ is a classical solution of \ref{eq:LS}, i.e. $z \in C^1([0,\infty);X) \cap C([0,\infty);X_1)$, see \cite[PART~II,~Chapter~2,~Proposition~3.3]{BenPratDefMit}. Therefore, we obtain
	\begin{align*}
		\MoveEqLeft \frac{\d}{\d t} \lVert z(t) \rVert_X^2\\
		&= 2 \real \scpr{\dot{z}(t)}{z(t)}{X} \\
		&= 2 \real \scpr{Az(t) + B_1 u_1(t) + B_2u_2(t)}{z(t)}{X}\\
		&=2 \real \left( \scpr{Az(t)}{z(t)}{X} + \scpr{ u_1(t)}{B_1^* z(t)}{U_1} + \scpr{ u_2(t)}{ B_2^*z(t)}{U_2} \right).
	\end{align*}
	Integration over $[0,t]$ yields
	\begin{align}\label{eq:proof_representation_norm(z)^2_LS_analytic}
		\begin{split}
			\MoveEqLeft \lVert z(t) \rVert_X^2 - \lVert z_0 \rVert_X^2 \\
			&=   2 \real \int_0^t \scpr{Az(s)}{z(s)}{X} + \scpr{u_1(s)}{B_1^*z(s)}{U_1} + \scpr{u_2(s)}{B_2^*z(s)}{U_2} \d s.
		\end{split}
	\end{align}
	Since $\scpr{Az(s)}{z(s)}{X} = \dualp{Az(s)}{z(s)}{X_{-\frac{1}{2}}}{X_\frac{1}{2}}$ for the classical solution $z$, we obtain from the previously mentioned continuous dependencies and density of $D(A)$ in $X$ and $\{ u \in \H^1(0,t;U_i) \mid u(0)=0 \}$ in $\L^2(0,t;U_i)$ for every $t \geq 0$ that \eqref{eq:proof_representation_norm(z)^2_LS_analytic} holds for all initial data $z_0 \in X$ and $u_i \in \L^2(0,\infty;U_i)$, $i=1,2$.\\
	
	Now, let $z_0 \in X$ and $u_1 \in \L^2_{\rm loc}(0,\infty;U_1)$ and denote by $z \in C([0,t_{\rm max});X)$ the corresponding solution from \Cref{lem:local_in_time_solution} on $[0,t_{\rm max})$. Changing $u_1$ on $[t_{\rm max},\infty)$ does not affect $z$, thus, we assume without loss of generality that $u_1 \in \L^2(0,\infty;U_1)$. Set $u_2=N(z,Cz) \in \L^2(0,t_{\rm max};U_2)$ and extend it by $0$ to a function in $\L^2(0,\infty;U_2)$. Hence, $z$ is also the restriction to $[0,t_{\rm max})$ of the solution of the linear system \ref{eq:LS} for $x_0$, $u_1$ and $u_2$. Consequently, $z$ has the desired regularity properties and \eqref{eq:proof_representation_norm(z)^2_LS_analytic} holds for $t \in [0,t_{\rm max})$ with $u_2=N(z,Cz)$, which completes the proof.
\end{proof}

\begin{lemma}\label{lem:representation_norm(z)^2_skew}
	Let $U_1,U_2,X,Y$ be Hilbert spaces. In addition to our standing assumptions, suppose that \Cref{assump:A_essentially_skew-adjoint} holds. For $z_0 \in X$ and $u_1 \in \L^2_{\rm loc}(0,\infty;U_1)$, let $t_1>0$ such that \ref{eq:BFS} admits a unique solution $z$ on $[0,t_1]$. For $t \in [0,t_1]$ it holds that
	\begin{align}\label{eq:representation_norm(z)^2_skew}
		\begin{split}
			\MoveEqLeft \lVert z(t) \rVert_X^2 - \lVert z_0 \rVert_X^2 \\
			&=  2 \real \int_0^t \scpr{Lz(s)}{z(s)}{X} + \scpr{u_1(s)}{B_1^*z(s)}{U_1} \\
			& \qquad \qquad \qquad \quad +\scpr{N(z(s),Cz(s))}{B_2^*z(s)}{U_2} \d s.
		\end{split}
	\end{align}
\end{lemma}

\begin{proof}
	Note  that since $A=A_0 + L$ with bounded $L$ and skew-adjoint $A_0$, one has $\real \scpr{Az}{z}{X} = \real \scpr{Lz}{z}{X}$ for all $z \in D(A)=D(A_0)$. With this at hand, it is not difficult to see that the proof is similar to the one for \Cref{lem:representation_norm(z)^2_analytic}. Indeed, we only have to ensure the continuous dependency of the solution of the linear system with respect to the initial data in $C([0,t];X)$. But this is clear from the boundedness of $B_1$, $B_2$ and $C$.
\end{proof}

If we consider input data with the smallness condition \eqref{eq:small_initial_data} from \Cref{thm:global_existence_for_small_data} in \Cref{lem:representation_norm(z)^2_analytic} or \Cref{lem:representation_norm(z)^2_skew}, we clearly obtain that $t_{\rm max}=\infty$. Under an additional dissipation condition on the nonlinear part, we  can get rid of the smallness condition to obtain $t_{\rm max} =\infty$ and global ISS results. This if formulated in the following theorem.

\begin{theorem}\label{thm:ISS_for_BFS}
	In addition to our standing assumptions, suppose that either \Cref{assump:A_self-adjoint} or \Cref{assump:A_essentially_skew-adjoint} holds and let $w_A<0$ such that $\real \scpr{Az}{z}{X} \leq w_A \lVert z \rVert_X^2$ for all $z \in D(A)$. If there exists $m_1,m_2 \in \R$ satisfying
	\begin{equation*}
		1-m_1 >0 \quad \text{ and } \quad (1-m_1)w_A + m_2 < 0
	\end{equation*}
	such that for all $z \in D(A)$ it holds that
	\begin{equation}\label{eq:decay_property_B2N}
		\lvert \scpr{N(z,Cz)}{B_2^*z}{U_2} \rvert \leq -m_1  \real \scpr{Az}{z}{X} + m_2 \lVert z \rVert_X^2,
	\end{equation}
	then there exist constants $c,\nu>0$ such that \ref{eq:BFS} admits for all $z_0 \in X$ and $u_1 \in \L^2_{\rm loc}(0,\infty;U_1)$ a unique solution $z$ which satisfies for all $t \geq 0$ that
	\begin{equation}\label{eq:local_L2_omega_ISS_estimmate}
		\lVert z(t) \rVert_X \leq \lVert z_0 \rVert_X \e^{- \nu t} + c \, \e^{- \nu t} \lVert \e^{\nu \cdot}u_1  \rVert_{\L^2(0,t;U_1)}.
	\end{equation}
	
	Moreover, \ref{eq:BFS} is $\L^q$-ISS for all $q \in [2,\infty]$.
\end{theorem}

\begin{proof}
	For any $x_0 \in X$ and $u_1 \in \L^2_{\rm loc}(0,\infty;U_1)$ there exists $t_1>0$ such that \ref{eq:BFS} admits on $[0,t_1]$ a unique solution $x \in C([0,t_1];X)$ and output $y\in \L^2(0,t_1;Y)$ by \Cref{lem:local_in_time_solution}. Moreover, it suffices to prove \eqref{eq:local_L2_omega_ISS_estimmate} on $[0,t_1]$. Indeed, then the solution stays bounded on its existence interval and by \Cref{lem:local_in_time_solution}, it can be extended to a solution on $[0,\infty)$. Note that this solution satisfies \eqref{eq:local_L2_omega_ISS_estimmate} on $[0,\infty)$.

	First consider \Cref{assump:A_self-adjoint}. From the regularity of $z$ derived in \Cref{lem:representation_norm(z)^2_analytic} and the continuous dependency of
	\eqref{eq:decay_property_B2N} on $z \in X_{\frac{1}{2}}$ it follows for almost every $t \geq 0$ that
	\begin{equation*}
		\real \scpr{N(z(t),Cz(t))}{B_2^*z(t)}{U_2} \leq -m_1 \real \dualp{A_{-1}z(t)}{z(t)}{X_{-\frac{1}{2}}}{X_\frac{1}{2}}  + m_2 \lVert z(t) \rVert_X^2.
	\end{equation*}
	\Cref{lem:representation_norm(z)^2_analytic} also yields that $\lVert z \rVert_X^2$ is almost everywhere differentiable and
	\begin{align}\label{eq:ineq_diff_gronwall_self-adjoint}
		\begin{split}
			\MoveEqLeft \frac{1}{2} \frac{\d}{\d t}\lVert z(t) \rVert_X^2 \\
			&= \real \Big( \dualp{A_{-1}z(t)}{z(t)}{X_{-\frac{1}{2}}}{X_{\frac{1}{2}}}  + \scpr{u_1(t)}{B_1^*z(t)}{U_1} \\
			& \qquad \qquad+ \scpr{N(z(t),Cz(t))}{B_2^*z(t)}{U_2} \Big)\\
			&\leq (1-m_1) (-\lVert z \rVert_{X_\frac{1}{2}}^2) + m_2 \lVert z(t) \rVert_X^2 \\
			& \qquad \qquad + \lVert B_1^* \rVert_{\Lb(X_\frac{1}{2},U_1)} \lVert u_1(t) \rVert_{U_1} \lVert z(t) \rVert_{X_{\frac{1}{2}}} \\
			&\leq [(1-m_1 -\mu) w_A + m_2] \lVert z(t) \rVert_X^2 + \frac{\lVert B_1^* \rVert_{\Lb(X_\frac{1}{2},U_1)}^2}{4 \mu} \lVert u_1(t) \rVert_{U_1}^2\\
		\end{split}
	\end{align}
	for $\mu >0$ such that $1-m_1-\mu>0$ and $- \nu \coloneqq [(1-m_1-\mu) w_A + m_2]<0 $, where we applied \eqref{eq:A_strictly_negative_extended} in the first inequality and Young's inequality with $q=q'=2$ in the last one. Gronwall's inequality yields that
	\begin{equation*}
		\lVert z(t) \rVert^2 \leq  \lVert z_0 \rVert_X^2 \e^{-2 \nu t} + \frac{\lVert B_1^* \rVert_{\Lb(X_\frac{1}{2},U)}^2}{2 \mu}\int_0^t  \lVert u(s) \rVert_{U_1}^2 \e^{-2 \nu (t-s)} \d s
	\end{equation*}
	and hence, the required estimate follows,
	\begin{equation*}
		\lVert z(t) \rVert_X \leq \lVert z_0 \rVert \e^{-\nu t} + \left(\frac{\lVert B_1^* \rVert_{\Lb(X_\frac{1}{2},U)}^2}{2 \mu}\right)^{\nicefrac{1}{2}}  \e^{-\nu t} \lVert u_1 \e^{\nu \cdot} \rVert_{\L^2(0,t;U_1)}.
	\end{equation*}
	
	If \Cref{assump:A_essentially_skew-adjoint} holds, we obtain an analog estimate to \eqref{eq:ineq_diff_gronwall_self-adjoint} by using \Cref{lem:representation_norm(z)^2_skew}, replacing $\dualp{A_{-1}z(t)}{z(t)}{X_{-\frac{1}{2}}}{X_\frac{1}{2}}$ by $\scpr{Lz(t)}{z(t)}{X}$, $\lVert B_1^*\rVert_{\Lb(X_{\frac{1}{2}},U)}$ by $\lVert B_1^* \rVert_{\Lb(X,U)}$ and by using the strict dissipativity of $L$ instead of \eqref{eq:A_strictly_negative_extended}. In this case, Gronwall's inequality yields that
	\begin{equation*}
		\lVert z(t) \rVert_X \leq \lVert z_0 \rVert \e^{-\nu t} + \left(\frac{\lVert B_1^* \rVert_{\Lb(X,U_1)}^2}{2 \mu}\right)^{\nicefrac{1}{2}}  \e^{-\nu t} \lVert u_1 \e^{\nu \cdot} \rVert_{\L^2(0,t;U_1)}.
	\end{equation*}
	
	That \ref{eq:BFS} is $\L^q$-ISS for any $q \in [2,\infty]$ follows from the fact that $\L^q(0,\infty;U_1) \subseteq \L^2_{\rm loc}(0,\infty;U_1)$ and H\"older's inequality, which implies for every $t \geq 0$ that
	\begin{equation}
		\e^{- \nu t} \lVert u_1 \e^{\nu \cdot} \rVert_{\L^2(0,t;U_1)} \leq c \lVert u_1 \rVert_{\L^q(0,t;U_1)}
	\end{equation}
	for some absolute constant $c>0$ independent of $t$.
\end{proof}

\begin{remark}
	In the proof of \Cref{thm:ISS_for_BFS} we proved that $V:X \rightarrow \R$, $V(z)= \lVert z \rVert_X^2$ is an ISS-Lyapunov function for \ref{eq:BFS} under the additional assumption \eqref{eq:decay_property_B2N}.  This assumption has been used in \cite{Schw20} to derive ISS estimates for parabolic semilinear boundary control systems with (time-depending) semilinearities mapping an interpolation space between $D(A)$ and $X$, related to the fractional power of the operator $-A$, boundedly into $X$. Compared to our setting, neither feedback nor unboundedness of the nonlinearity, in the sense of the presence of an unbounded operator $B_2$, are considered.
\end{remark}

\section{Examples}\label{section:examples}

\subsection{The Burgers equation}

Stability of the viscous Burger equation has been studied in several works, such as \cite{Krstic_BurgersBoundary, LyMeTi,Sachdev}, to name only a few of them. In \cite{ZhengZhu_Burgerseq}, local ISS with respect to $\L^\infty$ of a Burgers equation on the state space $\L^2(0,1)$ with in-domain and boundary controls/disturbances is proven under additional regularity assumptions on the controls/disturbances corresponding to the used solution concept of classical solutions.

We consider the following controlled viscous Burgers equation with Dirichlet boundary conditions for $t \in [0,\infty)$ and $x \in (0,1)$

\begin{equation}\label{eq:Burgers}
	\left\{\begin{aligned}
		\dot{z}(t,x) &= \frac{\partial^2z}{\partial x^2}(t,x) - z(t,x) \frac{\partial z}{\partial x}(t,x) + u_1(t,x),\\
		z(t,0)&=z(t,1)=0,\\
		z(0,x) &=z_0(x),\\
		y(t,x) &= z(t,x).
	\end{aligned}\right.
\end{equation}

We apply the results from \Cref{section:ISS_bilin_feedback} to the above Burgers equation, considered once on the state space $\H_0^1(0,1)$ and once on the state space $\L^2(0,1)$.

First, let the state, input and output spaces be given as in \cite[Section~8]{TucWei14}, i.e.
\begin{align}\label{eq:spaces_Burgers_eq_H1}
	\begin{split}
		X &= \H_0^1(0,1),\\
		U_1 &=U_2 = \L^2(0,1),\\
		Y&=\H^2(0,1) \cap H_0^1(0,1),
	\end{split}
\end{align}
where all spaces are assumed to be real valued. We equip $\H_0^1(0,1)$ with the norm
\begin{equation*}
	\lVert z \rVert_{\H_0^1(0,1)}= \lVert \tfrac{\d z}{\d x} \rVert_{\L^2(0,1)}.
\end{equation*}
It follows from the Poincaré inequality that this defines a norm which is equivalent to the standard norm on $\H_0^1(0,1)$.\\ 
Let the operator $A$ on $X$ be defined by
\begin{align*}
	A \varphi = \frac{\d^2 \varphi}{\d x^2}, \quad 
	D(A) = \{ \varphi \in \H^3(0,1) \mid \varphi, \tfrac{\d^2 \varphi}{\d x^2} \in \H_0^1(0,1) \}.
\end{align*}
It is known that $A$ is a self-adjoint and strictly negative.

Note that (see \cite[Section~8]{TucWei14} and the references therein)
\begin{equation*}
	X_{\frac{1}{2}} = \H^2(0,1) \cap \H_0^1(0,1) \quad \text{ and } \quad X_{-\frac{1}{2}}= \L^2(0,1).
\end{equation*}
Further, we consider the operators $B_i \in \Lb(U_i,X_{-\frac{1}{2}})$, $i=1,2$ and $C \in \Lb(X_\frac{1}{2},Y)$ to be the identity on the respective spaces. The bilinear feedback law is defined by
\begin{equation*}
	N:X \times Y \rightarrow U_2, \quad N(z,y) = -z \frac{\d y}{\d x}.
\end{equation*}
The validity of \eqref{eq:assumptions_on_N} for any $p \in (0,1)$ is not difficult to check and can also be found in \cite[Proof~of~Proposition~8.4]{TucWei14}. Thus, \Cref{thm:global_existence_for_small_data} together with \Cref{lem:representation_norm(z)^2_analytic} yields the following result

\begin{theorem}
	The Burgers equation \eqref{eq:Burgers} with spaces as in \eqref{eq:spaces_Burgers_eq_H1} and operators as above is a bilinear feedback system of the form \ref{eq:BFS}. Moreover, there exist $\omega, \varepsilon>0$ such that \eqref{eq:Burgers} admits for all $z_0 \in \H_0^1(0,1)$ and $u_1 \in \L^2_\omega(0,\infty;\L^2(0,1))$ with
	\begin{equation*}
		\lVert z_0 \rVert_{\H_0^1(0,1)} + \lVert u_1 \rVert_{\L^2_\omega(0,\infty;\L^2(0,1))} \leq \varepsilon
	\end{equation*}
	a unique solution $z \in \H^1(0,\infty;\L^2(0,1)) \cap C([0,\infty;\H_0^1(\Omega)) \cap \L^2(0,\infty;\H^2(0,1))$ which satisfies for some $k>0$ and every $t \geq 0$ that
	\begin{equation*}
		\lVert z(t) \rVert_X \leq k \lVert z_0 \rVert_{\H_0^1(0,1)} \e^{- \omega t} + k \e^{- \omega t } \lVert u_1 \e^{\omega \cdot} \rVert_{\L^2(0,t;\L^2(0,1))}.
	\end{equation*}
	In particular, \eqref{eq:Burgers} is locally $\L^2_\omega$-ISS.
\end{theorem}

\begin{remark}
	In \cite[Theorem~8.1]{TucWei14} the authors proved that the Burgers equation admits global solutions for all input data $z_0 \in \H_0^1(0,1)$ and $u_1 \in \L^2(0,\infty;\L^2(0,1))$. Unfortunately, \eqref{eq:decay_property_B2N} does not hold for all $z \in D(A)$, so our methods does not guarantee a global $\L^2$-ISS-estimate for input data which are not small if we consider the equation on the spaces from \eqref{eq:spaces_Burgers_eq_H1}.
\end{remark}

Now, let us consider \eqref{eq:Burgers} with the real-valued spaces
\begin{align}\label{eq:spaces_Burgers_eq_L2}
	\begin{split}
		X &= \L^2(0,1),\\
		U_1&=U_2 = \H^{-1}(0,1),\\
		Y &= H_0^1(0,1).
	\end{split}
\end{align}
Let $A$ be given by
\begin{align*}
	A\varphi = \frac{\d^2 \varphi}{\d x^2}, \quad D(A) &= \H^2(0,1) \cap \H_0^1(0,1).
\end{align*}
As before, $A$ is self-adjoint and strictly negative, and we obtain
\begin{equation*}
	X_{\frac{1}{2}} = H_0^1(0,1) \quad \text{and} \quad X_{-\frac{1}{2}} = \H^{-1}(0,1).
\end{equation*}
The operators $B_i \in \Lb(U_i,X_{-\frac{1}{2}})$ and $C \in \Lb(X_{\frac{1}{2}},Y)$ are considered to be the identity on the respective spaces. The bilinear feedback law given by
\begin{equation*}
	N :X \times Y \rightarrow U_2, \quad  N(z,y) = -\frac{1}{2} \frac{\d (zy)}{\d x} 
\end{equation*}
satisfies \eqref{eq:assumptions_on_N}. Indeed, for $z \in X$ and $y \in Y$, the continuity of the embedding $\H^{s}(0,1) \hookrightarrow C([0,1])$ for $s \in (\frac{1}{2},1)$, see e.g. \cite[Theorem~7.63]{Adams} or \cite[Theorem~8.2]{hitchhikersobolev}, and the classical interpolation result \cite[Corollary~1.7~\&~Example~1.10]{Lunardi} imply for $\alpha \in (0,\frac{1}{2})$ that
\begin{equation*}
	\lVert y \rVert_{C([0,1])} \leq K \lVert y \rVert_{\H^{\frac{1}{2} + \alpha }} \leq K \lVert y \rVert_{L^2(0,1)}^{1-p} \lVert y \rVert_{\H^1(0,1)}^p
\end{equation*}
with $p= \frac{1}{2} + \alpha \in (0,1)$ and hence,
\begin{equation*}
	\lVert N(z,y) \rVert_{U_2} \leq \frac{1}{2} \lVert zy \rVert_{L^2(0,1)} \leq \frac{1}{2} K \lVert z \rVert_{L^2(0,1)}  \lVert y \rVert_{L^2(0,1)}^{1-p} \lVert y \rVert_{\H^1(0,1)}^p.
\end{equation*}
For $z \in D(A)$ there holds that
\begin{equation*}
	\langle N(z,Cz),z \rangle_{\L^2(0,1)} = \frac{1}{3} \int_0^1  \frac{\d z^3}{\d x}(x) \d x = \frac{1}{3}(z(1)-z(0))=0,
\end{equation*}
which allows us to apply \Cref{thm:ISS_for_BFS}. Combing this with \Cref{lem:representation_norm(z)^2_analytic} yields:

\begin{theorem}\label{thm:Burgers_global}
	The Burgers equation \eqref{eq:Burgers} with spaces as in \eqref{eq:spaces_Burgers_eq_L2} and operators as above is a bilinear feedback system of the form \ref{eq:BFS}. Moreover, \eqref{eq:Burgers} admits for all $z_0 \in \L^2(0,1)$ and $u_1 \in \L^2(0,\infty;\H^{-1}(0,1))$ a unique solution $z \in \H^1(0,\infty;\H^{-1}(0,1)) \cap C([0,\infty);\L^2(0,1)) \cap \L^2(0,\infty;\H_0^1((0,1)))$, which satisfies for some $\nu,c>0$ and every $t \geq 0$ that
	\begin{equation*}
		\lVert z(t) \rVert_{\L^2(0,1)} \leq \lVert z_0 \rVert_{\L^2(0,1)} \e^{-\nu t} + c \e^{-\nu t} \lVert \e^{\nu \cdot} u_1 \rVert_{\L^2(0,t;\H^{-1}(0,1))}.
	\end{equation*}
	In particular, \eqref{eq:Burgers} is $\L^q$-ISS for every $q \in [2,\infty]$.
\end{theorem}

Such an ISS estimate for the Burgers equation on the state space $\L^2$ is not new, but we want to emphasize that the abstract operator-theoretic methods derived in this note also cover this type of examples.

\subsection{The Schr\"odinger equation}

We consider the following controlled Schrö\-din\-ger for $t \in[0,\infty)$ and $x \in (0,1)$

\begin{equation}\label{eq:Schrodinger}
	\left\{
	\begin{aligned}
		\dot{z}(t,x) &= i \frac{\partial^2 z}{\partial x^2}(t,x) - z(t,x) + (z(t,x))^2 + u_1(t,x),\\
		z(t,0) &=z(t,1)=0,\\
		z(0,x) &= z_0,\\
		y(t,x) &= z(t,x).
	\end{aligned}\right.
\end{equation}
We take the spaces as in \eqref{eq:spaces_Burgers_eq_H1}, which are here assumed to be complex valued. We define the operator $A: D(A) \subseteq X \rightarrow X$ by 
\begin{equation*}
	A \varphi = \iu \frac{\d^2 \varphi}{\d x^2} - \varphi, \quad \varphi \in D(A) = \left\{ \varphi \in \H^3(0,1) \mid \varphi, \tfrac{\d^2 \varphi}{\d x^2} \in \H_0^1(0,1) \right\}.\\
\end{equation*}
Hence, $A$ is the generator of an exponentially stable semigroup. With the input and output spaces $U_1=U_2=Y=X$ and the bounded operators $B_1=B_2=C=I$ the linear system \ref{eq:LS} becomes well-posed. Let $N:X \times Y \rightarrow U_2$ be defined by
\begin{equation*}
	N(z,y)= zy.
\end{equation*}
Thus, $N$ satisfies \eqref{eq:assumptions_on_N} for any $p \in (0,1)$.
Indeed, it follows from the continuous embedding $\H^1(0,1) \hookrightarrow C([0,1])$ with embedding constant $c>0$ that
\begin{equation*}
	\lVert N(z,y) \rVert_{\H^1(0,1)} \leq \lVert \tfrac{\d z}{\d x}y \lVert_{\L^2(0,1)} + \lVert z \tfrac{\d y}{\d x} \rVert_{\L^2(0,1)} \leq 2 c \lVert z \rVert_{\H^1(0,1)} \lVert y \rVert_{H^1(0,1)}.
\end{equation*}
Hence, \Cref{thm:global_existence_for_small_data} is applicable and we obtain the following.

\begin{theorem}
	The Schr\"odinger equation \eqref{eq:Schrodinger} is a bilinear feedback system of the form \ref{eq:BFS} with the above spaces and operators. Moreover, there exist $\omega,\varepsilon>0$ such that \eqref{eq:Schrodinger} admits for all $z_0 \in \H_0^1(0,1)$ and $u_1 \in \L^2_\omega(0,\infty;\L^2(0,1))$ with
	\begin{equation*}
		\lVert z_0 \rVert_{\H_0^1(0,1)} + \lVert u_1 \rVert_{\L^2_\omega(0,\infty;\L^2(0,1))} \leq \varepsilon,
	\end{equation*}
	a unique solution $z \in C([0,\infty);\H_0^1(0,1))$, which satisfies for some $k>0$ and every $t \geq 0$ that
	\begin{equation*}
		\lVert z(t) \rVert_{\H_0^1(0,1)} \leq k \lVert z_0 \rVert_{\H_0^1(0,1)} \e^{- \omega t} + k\e^{-\omega t} \lVert u_1 \e^{\omega \cdot} \rVert_{\L^2(0,t;\L^2(0,1))}.
	\end{equation*}
	In particular, \eqref{eq:Schrodinger} is locally $\L^2_\omega$-ISS.
\end{theorem}

\subsection{The Navier--Stokes equations}

The following example of the Navier--Stokes equations are taken from \cite{TucWei14}. There, the authors already considered the Navier--Stokes equations as a bilinear feedback system to prove local in time well-posedness with well-posedness methods. All operator theoretic statements used in this section can be found in \cite[Section~9]{TucWei14} and the references therein.\medskip

Let $\Omega \subseteq \R^3$ be a bounded domain with $C^2$ boundary $\partial \Omega$. The Navier--Stokes equations read for $t \in [0,\infty)$ and $x \in \Omega$

\begin{equation}\label{eq:Navier_Stokes}
	\left\{\begin{aligned}
		\rho \dot{z}(t,x) - \nu \Delta z(t,x) + \rho[(z \cdot \nabla)z](t,x) + \nabla p(t,x) &= u_1(t,x)\\
		\Div z(t,x) &= 0,\\
		z_{\vert_{x \in \partial \Omega}}(t,x) &= 0,\\
		z(0,x) &= z_0(x).
	\end{aligned}\right.
\end{equation}
The Navier--Stokes system describes the motion of an incompressible viscous fluid in the bounded domain $\Omega$. The Eulerian velocity field of the fluid $z$ and the pressure field in the fluid $p$ are unknown, while the density $\rho$ and $\nu$ the viscosity of the fluid  are given positive constants. By $P$ we denote the orthogonal projection from $\L^2(\Omega;\R^3)$ onto the closed subspace
\begin{equation*}
	\L^{2,\sigma}(\Omega) \coloneqq \{  \varphi \in \L^2(\Omega;\R^3) \mid \Div \varphi=0, \, \varphi \cdot \vec{n} =0 \text{ on } \partial \Omega\},
\end{equation*}
where $\vec{n}$ denotes the outward pointing unit normal vector on $\partial \Omega$ and $\varphi \cdot \vec{n} =0$ is understood in the weak sense, i.e. for every $\psi \in \H^1(\Omega)$ it holds that
\begin{equation*}
	\int_\Omega \varphi \cdot \nabla \psi \d x = 0.
\end{equation*}
In the literature, $\L^{2,\sigma}(\Omega)$ is often denoted by $\L^2_\sigma(\Omega)$. Our notation avoids confusion with the weighted $\L^2$ spaces $\L^2_\omega$, as introduced \Cref{section:bilinear_feedback_systems}. The projection $P$ is called the {\em Helmholtz} or {\em Leray projector}, and it is known that
\begin{equation*}
	G(\Omega) \coloneqq (I-P) \L^2(\Omega;\R^3)
\end{equation*}
can be given by $G(\Omega)= \nabla (\widehat{\H^1}(\Omega))$, where
\begin{equation*}
	\widehat{H^1}(\Omega) \coloneqq \left\{ q \in \H^1(\Omega) \middle \vert \int_\Omega q \d x =0 \right\}.
\end{equation*}
One can prove that $\nabla : \widehat{\H^1}(\Omega) \rightarrow G(\Omega)$ is a bounded invertible operator with bounded inverse, denoted by $\mathcal{M}$.\\
The {\em Stokes operator} $A_0$ is defined by
\begin{equation*}
	A_0 \varphi = - \frac{\nu}{\rho} P \Delta \varphi, \quad \varphi \in D(A_0)= \L^{2,\sigma}(\Omega) \cap \H_0^1(\Omega;\R^3) \cap \H^2(\Omega;\R^3).
\end{equation*}
It turns out that $A_0$ is a self-adjoint, strictly positive operator on $\L^{2,\sigma}(\Omega)$. Hence, we can define
\begin{equation*}
	A  \varphi = -A_0 \varphi, \quad \varphi \in D(A)=D(A_0^{\nicefrac{3}{2}}) 
\end{equation*}
on the state space 
\begin{equation*}
	X=D(A_0^{\nicefrac{1}{2}})=\{ \varphi \in \H_0^1(\Omega;\R^3) \mid \Div \varphi =0\}
\end{equation*}
equipped with the standard norm on $\H_0^1(\Omega;\R^3)$ which is equivalent to the graph norm of $A_0^{\nicefrac{1}{2}}$. Thus, $A$ is a self-adjoint and strictly negative operator on $X$ and we obtain that
\begin{equation*}
	X_{\frac{1}{2}}= D(A_0), \quad X_{-\frac{1}{2}} = \L^{2,\sigma}(\Omega).
\end{equation*}
Further, we introduce the spaces
\begin{equation*}
	U_1=U_2= \L^2(\Omega;\R^3), \quad Y= D(A_0)
\end{equation*}
and the operators $B_i \in \Lb(U_i,X_{-\frac{1}{2}})$, $i=1,2$ and $C \in \Lb(X_{\frac{1}{2}},Y)$ with bounded extension $C \in \Lb(X)$ given by
\begin{equation*}
	B_1=B_2= P, \quad C=I.
\end{equation*}
We define the bilinear mapping $N:X \times Y \rightarrow U_2$ by
\begin{equation*}
	N(z,y) = - [(z \cdot \nabla)y].
\end{equation*}
In \cite[Proof~of~Proposition~9.2]{TucWei14} it is shown that $N$ satisfies \eqref{eq:assumptions_on_N} for $p=\frac{4}{5}$.

\begin{theorem}
	The Helmholtz projected version of the Navier--Stokes system \eqref{eq:Navier_Stokes} is a bilinear feedback system of the form \ref{eq:BFS} with the above spaces and operators. Moreover, there exist $\omega,\varepsilon>0$ such that \eqref{eq:Navier_Stokes} admits for all $z_0 \in \H_0^1(\Omega;\R^3)$ with $\Div z_0 =0$ and $u_1 \in \L^2_\omega(0,\infty;\L^2(\Omega;\R^3))$ with
	\begin{equation*}
		\lVert z_0 \rVert_{\H_0^1(\Omega;\R^3)} + \lVert u_1 \rVert_{\L^2_\omega(0,\infty;\L^2(\Omega;\R^3))} \leq \varepsilon
	\end{equation*}
	a unique solution $(z,p)$,
	\begin{align*}
		z &\in  \H^1(0,\infty;\L^2(\Omega;\R^3)) \cap C([0,\infty);\H_0^1(\Omega;\R^3)) \cap \L^2(0,\infty;\H^2(\Omega;\R^3)) \\
		p &\in \L^2(0,\infty;\widehat{\H^1}(\Omega))
	\end{align*}
	which satisfies for some $k>0$ and every $t \geq 0$ that
	\begin{equation*}
		\lVert z(t) \rVert_{\H_0^1(\Omega;\R^3)} \leq k \lVert z_0 \rVert_{\H_0^1(\Omega;\R^3)} \e^{- \omega t } + k \e^{- \omega t } \lVert u_1 \e^{\omega \cdot} \rVert_{\L^2(0,\infty;\L^2(\Omega;\R^3))}.
	\end{equation*}
	In particular, \eqref{eq:Navier_Stokes} is locally $\L^2_\omega$-ISS.
\end{theorem}

\begin{proof}
	The assertion follows from the computations in the proof of \cite[Thorem~9.1]{TucWei14}, \Cref{thm:global_existence_for_small_data} and \Cref{lem:representation_norm(z)^2_analytic}. We give the details for the sake of completeness.\\
	Since the projected version of \eqref{eq:Navier_Stokes} is a bilinear feedback system for which \Cref{assump:A_self-adjoint} is satisfied, we obtain from \Cref{thm:global_existence_for_small_data} and \Cref{lem:representation_norm(z)^2_analytic} the existence of $\omega,\varepsilon>0$ such that there exists for every $z_0 \in X= \{ \varphi \in \H_0^1(\Omega;\R^3) \mid \Div \varphi =0\}$ and $u_1 \in \L^2_\omega(0,\infty;U_1)= \L^2_\omega(0,\infty;\L^2(\Omega;\R^3))$ satisfying
	\begin{equation*}
		\lVert z_0 \rVert_{\H_0^1(\Omega;\R^3)} + \lVert u_1 \rVert_{\L^2_\omega(0,\infty;\L^2(\Omega;\R^3)} \leq \varepsilon,
	\end{equation*}
	a unique solution
	\begin{equation*}
		z \in \H^1(0,\infty;\L^{2,\sigma}(\Omega)) \cap C([0,\infty);\H_0^1(\Omega;\R^3)) \cap \L^2(0,\infty);\H^2(\Omega)),
	\end{equation*}
	which satisfies
	\begin{equation*}
		\lVert z(t) \rVert_{\H_0^1(\Omega)} \leq k\lVert z_0 \rVert_{\H_0^1(\Omega)} \e^{-\omega t} + k \e^{-\omega t} \lVert u_1 \e^{\omega \cdot} \rVert_{\L^2(0,t;\L^2(\Omega;\R^3))}
	\end{equation*}
	for every $t \geq 0$ and some constant $k>0$.
	Since $z \in \H^1(0,\infty;\L^{2,\sigma}(\Omega))$, it holds that
	\begin{equation*}
		\rho \dot{z}(t) = \nu P\Delta z(t) - \rho P[(z(t) \cdot \nabla) z(t)] 
	\end{equation*}
	with each term in $\L^2(0,\infty;\L^2(\Omega;\R^3))$ and hence
	\begin{align*}
		\rho \dot{z}(t)
		&= \nu \Delta z(t) - \rho[(z(t) \cdot \nabla)z(t)] + u_1(t) \\
		& \quad - (I-P) [ \nu \Delta z(t) - \rho(z(t) \cdot \nabla)z(t) + u_1(t)].
	\end{align*}
	Since $(I-P) [ \nu \Delta z(t) - \rho(z(t) \cdot \nabla)z(t) + u_1(t)] \in \L^2(0,\infty;G(\Omega))$ by definition of $G(\Omega)$ we have that
	\begin{equation*}
		p(t) \coloneqq \mathcal{M} (I-P) [ \nu \Delta z(t) - \rho(z(t) \cdot \nabla)z(t) + u_1(t)] \in \L^2(0,\infty;\widehat{\H^1}(\Omega)).
	\end{equation*}
	It follows that the pair $(z,p)$ is a solution to \eqref{eq:Navier_Stokes} with the asserted regularities. The uniqueness follows by applying $P$ to \eqref{eq:Navier_Stokes} and the uniqueness of the solution of the projected version of the Navier--Stokes system, which is a bilinear feedback system.
\end{proof}

\subsection{A wave equation}\label{section:wave_equation}

We consider the following wave-type equation on a bounded Lipschitz domain $\Omega \subseteq \R^d$ with dimension $d \leq 4$. The equation reads, for $t \in[0,\infty)$ and $x \in \Omega$,

\begin{equation}\label{eq:wave_equation}
	\left\{ \begin{aligned}
		\ddot{\omega}(t,x) + \dot{\omega}(t,x) - \omega_{xx}(t,x) +(\omega(t,x))^2 &= u_1(t,x),\\
		\omega_{\vert_{x \in \partial \Omega}}(t,x) &=0,\\
		\omega(0,x)&=\omega_0(x),\\
		\dot{\omega}(0,x) &= \omega_1(x), \\
		y_1(t,x)&=\omega(t,x),\\
		y_2(t,x) &= \dot{\omega}(t,x),
	\end{aligned}\right.
\end{equation}
and the corresponding first order system on the state space $X= \H_0^1(\Omega) \times \L^2(\Omega)$
\begin{equation}\label{eq:wave_equation_first_order}
	\left\{ \begin{array}{rll}
		\begin{pmatrix}
			\dot{\varphi}(t)\\
			\dot{\psi}(t)\\
		\end{pmatrix} &= A \begin{pmatrix}
			\varphi(t)\\
			\psi(t)\\
		\end{pmatrix} + \begin{pmatrix}
			0 \\
			u_1(t)\\
		\end{pmatrix} + \begin{pmatrix}
			0\\
			\varphi^2(t)\\
		\end{pmatrix}, & t \in (0,\infty),\\
		\begin{pmatrix}
			\varphi(0)\\
			\psi(0)\\
		\end{pmatrix} &= \begin{pmatrix}
			\varphi_0\\
			\psi_0\\
		\end{pmatrix},\\
		\begin{pmatrix}
			y_1(t) \\
			y_2(t) \\
		\end{pmatrix} &= \begin{pmatrix}
			\varphi(t)\\
			\psi(t) \\
		\end{pmatrix} , & t \in [0,\infty),
	\end{array} \right.
\end{equation}
derived by the transformations
\begin{align*}
	\begin{pmatrix}
		\varphi \\
		\psi\\
	\end{pmatrix}
	= \begin{pmatrix}
		\omega \\
		\dot{\omega}\\
	\end{pmatrix},
	\quad  \begin{pmatrix}
		\varphi_0 \\
		\psi_0\\
	\end{pmatrix}
	= \begin{pmatrix}
		\omega_0\\
		\omega_1\\
	\end{pmatrix},
\end{align*}
where the operator $A$ on $X$ is given by
\begin{align*}
	A=
	\begin{pmatrix}
		0 & I \\
		\Delta & -I\\
	\end{pmatrix}, \quad D(A)= \H^2(\Omega) \cap \H_0^1(\Omega) \times \H_0^1(\Omega).
\end{align*}
It is well-known that $A$ generates an exponentially stable $C_0$-semigroup on $X$.
We define the input spaces and control operators
\begin{equation}
	U_1=U_2 = \L^2(\Omega), \quad B_1 u =B_2 u=
	\begin{pmatrix}
		0\\
		u\\
	\end{pmatrix},
\end{equation}
as well as the output space and observation operator
\begin{equation*}
	Y=X, \quad C=I.
\end{equation*}
It is obvious that $B_i$, $i=1,2$, are bounded from $U_i$ to $X$ and that $C$ is bounded from $X$ to $Y$. With these spaces and operators, \ref{eq:LS} is well-posed and \eqref{eq:wave_equation_first_order} can be seen as a bilinear feedback system \ref{eq:BFS} with bilinearity $N:X \times Y \rightarrow U_2$ given by
\begin{equation*}
	N\left(
	\begin{pmatrix}
		\varphi_1 \\
		\psi_1\\
	\end{pmatrix},
	\begin{pmatrix}
		\varphi_2\\
		\psi_2\\
	\end{pmatrix} \right) = \varphi_1 \varphi_2.
\end{equation*}

Since $d \leq 4$, the embedding $\H^1(\Omega) \hookrightarrow \L^4(\Omega)$ is continuous, which yields that $N$ is well-defined and satisfies \eqref{eq:assumptions_on_N} for any $p \in (0,1)$ (since $Y=X$). Indeed, for $\varphi_1 , \varphi_2 \in \H^1(\Omega)$ we have that
\begin{equation*}
	\lVert \varphi_1 \varphi_2 \rVert_{\L^2(\Omega)} \leq \lVert \varphi_1 \rVert_{\L^4(\Omega)} \lVert \varphi_2 \rVert_{\L^4(\Omega)} \leq c^2 \lVert \varphi_1 \rVert_{\H^1(\Omega)} \lVert \varphi_2 \rVert_{\H^1(\Omega)},
\end{equation*}
where $c$ is the embedding constant.

\Cref{thm:global_existence_for_small_data} implies the following result.

\begin{theorem}
	The wave equation \eqref{eq:wave_equation_first_order} is a bilinear feedback system of the form \ref{eq:BFS} with the above spaces and operators. Moreover, there exist $\omega,\varepsilon>0$ such that \eqref{eq:wave_equation_first_order} admits for all $(\varphi_0,\psi_0) \in \H_0^1(\Omega) \times \L^2(\Omega)$ and $u_1 \in \L^2_\omega(0,\infty;\L^2(\Omega))$ with
	\begin{equation*}
		\lVert \varphi_0 \rVert_{\H_0^1(\Omega)} + \lVert \psi_0 \rVert_{\L^2(\Omega)} + \lVert u_1 \rVert_{\L^2_\omega(0,\infty;\L^2(\Omega))} \leq \varepsilon,
	\end{equation*}
	a unique mild solution $(\varphi, \psi) \in C([0,\infty);\H_0^1(\Omega) \times \L^2(\Omega))$ which satisfies for some $k>0$ and every $t \geq 0$ that
	\begin{align*}
		&\lVert \varphi(t) \rVert_{\H^1(\Omega)} + \lVert \psi(t) \rVert_{\L^2(\Omega)} \\
		&\leq k \e^{- \omega t } \left( \lVert \varphi_0 \rVert_{\H^1(\Omega)} +  \lVert \psi_0 \rVert_{\L^2(\Omega)}  \right) + k \e^{-\omega t}  \lVert u_1 \e^{\omega \cdot} \rVert_{\L^2(0,t;\L^2(\Omega))}.
	\end{align*}
	In particular, \eqref{eq:wave_equation_first_order} is locally $\L^2_\omega$-ISS.
\end{theorem}


\def\cprime{$'$}



\end{document}